\documentclass[12pt]{article}
\usepackage[utf8]{inputenc}
\usepackage{amsmath}            
\usepackage{amssymb}          
\usepackage{amsfonts}           
\usepackage{mathrsfs}          
\usepackage{amsthm}
\usepackage{mathtools}
\usepackage{commath}
\usepackage{bm}
\usepackage{graphicx}
\usepackage{geometry}
\usepackage{float}
\usepackage{listings}
\usepackage{subfigure}
\usepackage{multirow}
\usepackage{diagbox}
\usepackage[export]{adjustbox}
\usepackage[all]{xy}
\usepackage{fancyhdr}
\usepackage{dsfont}
\usepackage{hyperref}
\usepackage{enumitem}
\usepackage{stmaryrd}
\usepackage{tikz-cd}
\usepackage{bbm}
\DeclareMathAlphabet{\mathbbb}{U}{bbold}{m}{n}
\usepackage{bold-extra}
\usepackage[T1]{fontenc}

\makeatletter
\newcommand*\bigcdot{\mathpalette\bigcdot@{.5}}
\newcommand*\bigcdot@[2]{\mathbin{\vcenter{\hbox{\scalebox{#2}{$\m@th#1\bullet$}}}}}
\makeatother

\numberwithin{equation}{section}

\theoremstyle{definition}
\newtheorem{defi}{Definition}[section]

\theoremstyle{plain}
\newtheorem{cor}[defi]{Corollary}
\newtheorem{cor-defi}[defi]{Corollary-Definition}
\newtheorem{lemma}[defi]{Lemma}

\newtheorem{prop-defi}[defi]{Proposition-Definition}
\newtheorem{theo}[defi]{Theorem}
\newtheorem{ques}[defi]{Question}

\theoremstyle{remark}
\newtheorem*{rmk}{Remark}

\title{$\delta$-lifting and $1$-dimensional analytic fields}
\author{Jiahong Yu}

\newcommand\dA{\mathbb{A}}

\newcommand\dF{\mathbb{F}}

\newcommand\dQ{\mathbb{Q}}
\newcommand\dR{\mathbb{R}}
\newcommand\dZ{\mathbb{Z}}

\newcommand\calE{\mathcal{E}}

\newcommand\calK{\mathcal{K}}

\newcommand\calO{{\mathcal{O}}}

\newcommand\calV{{\mathcal{V}}}

\newcommand\fraka{\mathfrak{a}}
\newcommand{\frakm}{\mathfrak{m}}
\newcommand{\frako}{\mathfrak{o}}
\newcommand\hol{\text{\normalfont{H}}}

\DeclareMathOperator{\Spf}{Spf}

\newcommand\an{\text{\normalfont{an}}}

\newcommand\id{\text{\normalfont{id}}}

\newcommand\AAinf{\dA_{\inf}}

\newcommand\Ainf{A_{\inf}}

\renewcommand\dif{\text{\normalfont{d}}}

\begin{document}

\maketitle

\begin{abstract}
    Let $k$ be an algebraically closed complete non-Archimedean field, and let $K$ be a finitely generated field extension over $k$ with transcendence degree $1$. Equip $K$ a non-Archimedean norm extending the one on $k$, and let $\calK$ denote the completion of $K$. We will prove that the valuation ring $\calK^+$ admits a flat $\delta$-lifting over $\AAinf(k^+)$ if and only if $\calK$ is not of type 4. 
\end{abstract}

\tableofcontents

\section{Introduction}

Throughout the paper, fix a prime number $p$.

The research of modern algebraic geometry and number theory heavily rely on the understanding of various mathematical objects over complete non-Archimedean fields. Such fields, particularly when their residue characteristic is $p$, form the foundation of rigid analytic geometry and $p$-adic Hodge theory. Investigating the algebraic extensions of these fields and the structure of their valuation rings is key to a deeper understanding of local arithmetic and geometric properties.

\subsection{Motivation from $p$-adic Hodge Theory: Prismatization of Geometric Valuation Rings}

Prismatic theory is a core framework that has emerged in recent years in the fields of $p$-adic geometry and $p$-adic Hodge theory, providing a powerful new tool for unifying the understanding of various $p$-adic cohomology theories.

A key contribution of this theory is the prismatic cohomology introduced by Bhatt--Scholze (cf. \cite{Bhatt_2022}). This theory introduces a novel algebraic structure called a "prism", which enables the description of several important $p$-adic cohomology theories, such as crystalline cohomology and \'etale cohomology, within a unified framework. Prismatic cohomology is regarded as a motivic $p$-adic cohomology theory, from which other known theories can be naturally derived.

Building on this, Bhatt and Lurie further developed the theories of "prismatization" and "absolute prismatic cohomology" (cf. \cite{bhatt2022absolute}, \cite{bhatt2022prismatization}). This work is an important generalization of the relative version of prismatic cohomology, aiming to construct more intrinsic and universal homological invariants for $p$-adic geometric objects. Prismatization can help us gain new insights into the deep arithmetic-geometric structures of $p$-adic varieties, and thus the precise computation of prismatization is a fundamental problem in modern $p$-adic Hodge theory.

Specifically, let $k$ be an algebraically closed complete non-Archimedean field whose residue field is of characteristic $p$, and let $\frako$ be its valuation ring. In \cite{bhatt2022prismatization}, the authors proved that for a $p$-completely smooth algebra $R$ over $\frako$, the Hodge--Tate locus of the prismatization of $\mathrm{Spf}(R)$ (aka. Hodge--Tate stack) is non-canonically isomorphic to the classifying space $BT_{R}^{\sharp}$ of an fpqc-group $T_{R}^{\sharp}$. Furthermore, when we fix a smooth $\delta$-lifting of $R/\frako$ over $\AAinf(\frako)$, we can canonically construct an isomorphism between the Hodge--Tate stack of $\Spf(R)$ and $BT_{R}^{\sharp}$. On the other hand, in \cite{qu2025rationalhodgetateprismaticcrystals}, the authors defined geometric valuation rings over $\frako$, as an analogue of valuation rings for smooth algebras over $\frako$. Therefore, we pose the following question:

\begin{ques}\label{ques: bhatt--lurie of val ring}
	To which geometric valuation rings can the computational method of Bhatt--Lurie be analogized?
\end{ques}

Specifically, let $\mathcal{K}^+$ be a geometric valuation ring over $\frako$. To analogize the Bhatt--Lurie computation, we first need to find a flat $\delta$-lifting of $\mathcal{K}^+$ over $\AAinf(\frako)$. Therefore, as a subquestion, we can pose the following problem:

\begin{ques}\label{ques: delta lift of val ring}
	For which geometric valuation rings $\mathcal{K}^+$ does $\mathcal{K}^+$ admit a flat $\delta$-lifting over $\AAinf(\frako)$?
\end{ques}

\subsection{Main results: Answering Question \ref{ques: delta lift of val ring} for one dimensional analytic field}\label{subsec: main 1}

In this paper, we study the points on a Berkovich curve over an algebraically closed field. More precisely, we work within the following framework.

Let $k$ be an algebraically closed, complete non-Archimedean field whose residue field is of characteristic $p$. Denote by $\frako$ the valuation ring of $k$, and let $A_{\inf}=\AAinf(\frako)$ be the Fontaine's period ring. Fix a generator $\xi$ of the kernel of Fontaine's $\theta$-map.

Let $\mathcal{K}/k$ be an extension of complete non-Archimedean fields, and make the following definition.

\begin{defi}\label{defi:1 dim an field}
    Such a $\calK$ is called a \emph{one-dimensional analytic extension over $k$} if there exists an element $z\in\calK$ such that $\calK$ is finite over the closure of $k(z)$.
\end{defi}
Let $\mathcal{K}^+$ be the valuation ring of $\mathcal{K}$.
The theory of Berkovich spaces classifies $\calK$ into the following types.

\begin{defi}\label{defi: Berkovich classify 1 dim geo rings}
    A one-dimensional analytic field $\calK/k$ is called:
    \begin{enumerate}
        \item \emph{type 2} if the residue field of $\calK^+$ is a proper extension of the residue field of $k$;
        \item \emph{type 3} if the value group of $\calK$ is a proper extension of the value group of $k$
        \item \emph{type 4} if none of the above conditions hold.
    \end{enumerate}
\end{defi}

By the Abhyankar inequality (see, for example, \cite[Appendix 2]{Zariski_1960}), every one-dimensional analytic field $\calK/k$ falls into exactly one of the classes defined above. Moreover, the Abhyankar inequality is an equality if and only if $\calK$ is not of type 4. In this case, we say that $\calK$ is \emph{Abhyankar}.

Now we state the first main theorem.

\begin{theo}[Subsection \ref{pr to main 1}]\label{main theo in intro}
    With the notation as above, the following are equivalent:
    \begin{enumerate}[label=(\roman*)]
        \item The extension $\mathcal{K}/k$ is Abhyankar;
        \item $\calK^+$ admits a $\delta$-lifting $(\widetilde{\calK}^+,\delta)$ over $\Ainf$ such that the structure map $\Ainf\to \widetilde{\calK}^+$ is $(p,\xi)$-completely flat.
    \end{enumerate}
    If in addition $\mathrm{char}(k)=p$, the above conditions are further equivalent to $\calK^+$ being $F$-split.
\end{theo}

Furthermore, we find that the type 2 points can be characterized by the $F$-finiteness of $\calK^+/p$.

\begin{theo}[Subsection \ref{pf to main 2}]\label{theo: main 2 f finite}
    With the notation as above, the following are equivalent:
    \begin{enumerate}[label=(\roman*)]
        \item The extension $\mathcal{K}/k$ is of type 2;
        \item The Frobenius map $\mathcal{K}^+/p\to \calK^+/p$ is finite.
    \end{enumerate}
\end{theo}

In the paper \cite{Datta_2016}, the authors studied the $F$-splitness and the $F$-finiteness of valuation rings of an algebraic function field. They claimed that, under their assumptions, the Abhyankar property implies $F$-finiteness (cf. \cite[Theorem 5.1]{Datta_2016}). However, they subsequently identified an error in this proof (cf. \cite{Datta_2017}). In \cite{Datta_2021}, it was proved that under the hypotheses of \cite[Theorem 5.1]{Datta_2016}, the Abhyankar property implies $F$-splitness. In fact, the proof of Theorem \ref{theo: main 2 f finite} provides an explicit counterexample to \cite[Theorem 5.1]{Datta_2016}.

\subsection{Acknowledgement}

I would like to express my sincere gratitude to Xiaoyu Qu and Shiji Lyu for their invaluable insights and stimulating discussions, which have greatly contributed to the development of this work. I am also deeply thankful to Tian Qiu for his careful reading of the draft and his constructive suggestions.

\newpage
\section*{Notations}
We introduce some notations we will frequently use in the paper.

\begin{enumerate}[label=(\arabic*)]
    \item  For any ring $R$ of characteristic $p$, denote by $F_R$ the Frobenius of $R$. In addition, for any $R$-module $M$, denote by $F_*M$ the Frobenius push-forward of $M$. That is, $F_*M$ is the $R$-module with underlying abelian group $M$ but the scalar product is defined as
\[a\cdot m=a^pm\ \forall a\in R,m\in M.\]

\item By a non-Archimedean field we mean a field $K$ with a non-Archimedean norm $|\bigcdot|:K\to \dR_{\geq 0}$. For any non-Archimedean field, we will use $K^+$ to denote the ring of integers of $K$. Also, for any $r>0$ and $x\in K$, denote by $B_K(x,r)$ the closed ball $\{y\in K:|y-x|\le r\}$.
\item For any ring $R$ of characteristic $p$, denote by $W_n(R)$ the $n$-truncated Witt ring of $R$. Denote by $W(R)$ the Witt ring of $R$.
\item  Let $A$ be a ring and $\fraka\subseteq A$ be an ideal. Denote by $\sqrt{\fraka}$ the nilradical of $\fraka$.
\end{enumerate}
\newpage
\section{Preliminaries}

This section reviews the preliminary material necessary for the proofs of Theorems \ref{main theo in intro} and \ref{theo: main 2 f finite}.

\subsection{General theory of valuation rings}

\noindent\textbf{Commutative algebra of valuation rings:} We begin by recalling some fundamental facts about valuation rings. Throughout this paper, we focus exclusively on the following class.

\begin{defi}
    A \emph{Tate valuation ring} is a valuation ring $\calO$ for which there exists a nonzero element $\pi\in \calO$ (called a \emph{pseudo-uniformizer}) satisfying that $\calO$ is $\pi$-adically complete.
    
    An \emph{extension of Tate valuation ring} is an inclusion of valuation rings $\calO\subset \calO'$ such that there exists a $\pi\in\calO$ that serves as a pseudo-uniformizer for both $\calO$ and $\calO'$.
\end{defi}

From now on, we fix a Tate valuation ring $\calO$ and a pseudo-uniformizer $\pi$. Moreover, we topologize $\calK=\calO[\frac{1}{\pi}]$ so that $\calO$ is open and is endowed with the $\pi$-adic topology.

The following lemma is fundamental.

\begin{lemma}\label{lemma: val rings radical}
    Let $x\in \calO$ be a nonzero element. There exists an integer $n\ge 1$ such that
    \[\pi^n\in x\calO.\]
    In other words, $\pi$ lies in the radical of any nonzero ideal.
\end{lemma}

\begin{proof}
    Since $\calO$ is $\pi$-adically complete, $\bigcap_{n\geq 1}\pi^n\calO=0$. In particular, there exists an integer $n\geq 1$ such that $x\notin \pi^n\calO$. Since $\calO$ is a valuation ring, for any two elements, one divides the other. Thus, $\pi^n\in x\calO$.
\end{proof}

This lemma immediately yields a criterion for recognizing pseudo-uniformizers.

\begin{cor}\label{cor: set of psu}
    A nonzero element $\pi'\in\calO$ is a pseudo-uniformizer if and only if there exists $n\geq 1$ such that
    \[\pi'^n\in\pi\calO.\]
\end{cor}

\begin{proof}
    If $\pi'$ is a pseudo-uniformizer, then applying Lemma \ref{lemma: val rings radical} by taking $x=\pi$ an $\pi=\pi'$ shows 
    \[\pi'^n\in\pi\calO\]
    for some $n\geq 1$.

    Conversely, if $\pi'^n\in\pi\calO$ for some $n\geq 1$, by Lemma \ref{lemma: val rings radical}, the $\pi$-adic topology on $\calO$ coincides with the $\pi'$-adic topology on $\calO$. Thus, $\pi'$ is a pseudo-uniformizer by definition.
\end{proof}

In addition, we establish basic topological properties of the fractional field.

\begin{cor}\label{cor: val field is O[1/pi] and val topology is pi-adic}
    The following hold:
    \begin{enumerate}[label=(\roman*)]
        \item The ring $\calK$ is the fractional field of $\calO$;
        \item The topology on $\calK$ is equal to the valuation topology defined by $\calO$.
    \end{enumerate}
\end{cor}

\begin{proof}
    Let $0\neq x\in \calO$. By Lemma \ref{lemma: val rings radical}, there exists an integer $N\ge 1$ such that $x^{-1}\pi^N\in \calO$. Thus, $x^{-1}\in \calO[\frac{1}{\pi}] \calK$, proving that $\calK$ is the fractional field of $\calO$.

    Regarding the topology, the collection $\{x\calO:0\neq x\in \calO\}$ forms a fundamental system of neighborhoods of $0$ in the valuation topology This implies that the valuation topology is finer than the $\pi$-adic topology on $\calO$. Conversely, Lemma \ref{lemma: val rings radical} shows that each nonzero principal ideal $x\calO$ contains $\pi^n\calO$ for some $n$. Hence, the $\pi$-adic topology is also finer than the valuation topology, and the two topologies coincide.
\end{proof}

The following lemma provides a simple torsion-theoretic criterion of an $\calO$-module being flat.

\begin{lemma}\label{lemma: val rings flat}
    Let $M$ be an $\calO$-module. Then $M$ is flat if and only if $M$ is $\pi$-torsion free.
\end{lemma}

\begin{proof}
    By \cite[\href{https://stacks.math.columbia.edu/tag/0539}{Tag 0539}]{stacks-project}, $M$ is flat if and only if for any $0\neq a\in \calO$, $M$ is $a$-torsion free. In particular, if $M$ is flat, then it is $\pi$-torsion free.

    Conversely, suppose $M$ is $\pi$-torsion free. Then for any integer $n\geq 1$, $M$ is $\pi^n$-torsion free. 
    Now, let $0\neq a\in\calO$ and suppose $am=0$ for some $m\in M$. By Lemma \ref{lemma: val rings radical}, there exists $n\geq 1$ such that $\pi^n\in a\calO$, so we may write $\pi^n=ab$ for some $0\neq b\in\calO$. Then $\pi^nm=b(am)=0$, and since $M$ is $\pi^n$-torsion free, it follows that $m=0$. Therefore, multiplication by $a$ is injective on $M$, and $M$ is flat. 
\end{proof}

Using this criterion, we can immediately deduce that flatness is preserved under completion.

\begin{cor}\label{cor: val ring cpt of flat is flat}
    If $M$ is a flat $\calO$-module, then its $\pi$-adic completion $\widehat{M}$ is also flat over $\calO$.
\end{cor}

\begin{proof}
    By Lemma \ref{lemma: val rings flat}, the flatness of $M$ implies that it is $\pi$-torsion free. Its $\pi$-adic completion $\widehat{M}$ is then also $\pi$-torsion free. Applying Lemma \ref{lemma: val rings flat} again shows that $\widehat{M}$ is flat over $\mathcal{O}$.
\end{proof}

It is a standard fact that in a valuation ring, an ideal is prime if and only if it is radical. Let $\mathcal{O}'$ be the localization of $\mathcal{O}$ at the prime ideal $\sqrt{\pi\mathcal{O}}$.

\begin{lemma}
    The ring $\mathcal{O}'$ is an open and bounded subset of $\mathcal{K}$. Moreover, there exists a norm $|\bigcdot| : \mathcal{K}^\times \to \mathbb{R}$ with respect to which $\mathcal{K}$ is a complete non-Archimedean field, and the topology on $\mathcal{K}$ coincides with the topology defined by this norm.
\end{lemma}

\begin{proof}
    Since $\mathcal{O}'$ is a localization of $\mathcal{O}$, it is also a valuation ring and contains $\mathcal{O}$. For any $x \in \mathcal{O}' \setminus \mathcal{O}$, we have $x^{-1} \in \mathcal{O} \subseteq \mathcal{O}'$ because $\mathcal{O}$ is a valuation ring. This shows that the maximal ideal of $\mathcal{O}$ is contained in $\mathcal{O}'$, and hence $\pi \mathcal{O}' \subseteq \mathcal{O}$. Therefore, $\mathcal{O}'$ is bounded. Its openness follows from the inclusion $\mathcal{O} \subseteq \mathcal{O}'$.
    
    It remains to show that $\mathcal{O}'$ has rank one, which is equivalent to its Krull dimension being one. By Lemma \ref{lemma: val rings radical}, every nonzero prime ideal of $\mathcal{O}$ contains $\sqrt{\pi\mathcal{O}}$. Thus, localizing at $\sqrt{\pi\mathcal{O}}$ yields a ring of Krull dimension one.
\end{proof}

\noindent\textbf{Complete differential modules}

In this paper, we will focus exclusively on the theory of complete differential modules in the context of extensions of Tate valuation rings. Let us fix a Tate valuation ring $\mathcal{O}$ with a pseudo-uniformizer $\pi \in \mathcal{O}$, and let $\mathcal{K} = \mathcal{O}[\frac{1}{\pi}]$ denote its fraction field.

For simplicity, assume that $\calO$ is $p$-adically complete.

\begin{defi}\label{defi: cpt cotangent cpt dif}
    Let $\mathcal{O}'/\mathcal{O}$ be an extension of Tate valuation rings. We define the \emph{complete cotangent complex} of $\mathcal{O}'/\mathcal{O}$ as the $\pi$-adic derived completion of the usual cotangent complex $\mathbb{L}_{\mathcal{O}'/\mathcal{O}}$, and denote it by $\widehat{\mathbb{L}}_{\mathcal{O}'/\mathcal{O}}$. By Corollary \ref{cor: set of psu}, this definition is independent of the choice of $\pi$.

    Furthermore, we define the \emph{complete differential module} of the extension as
    \[
    \widehat{\Omega}_{\mathcal{O}'/\mathcal{O}} = \hol^0\bigl(\widehat{\mathbb{L}}_{\mathcal{O}'/\mathcal{O}}\bigr),
    \]
    i.e., the zeroth cohomology of the complete cotangent complex.
\end{defi}

We now state several key properties of these objects that will be used in the paper.

\begin{theo}\label{theo: cpt diff mod of val rings}
    If $\mathcal{O}$ is perfectoid (in the sense of \cite[Definition 3.5]{bhatt2018integral}), then $\widehat{\mathbb{L}}_{\mathcal{O}'/\mathcal{O}} \simeq \widehat{\Omega}_{\mathcal{O}'/\mathcal{O}}[0]$, and the latter is a flat $\mathcal{O}'$-module.
\end{theo}

\begin{proof}
    In positive characteristic, the result is precisely \cite[Theorem 3.4]{Bouis_2023}.

    Now assume $\calO$ is of mixed characteristic $(0,p)$. By {\cite[Theorem 3.1]{Bouis_2023}} and the derievd Nakayama's lemma, we have $\widehat{\mathbb{L}}_{\mathcal{O}'/\mathcal{O}} \simeq \widehat{\Omega}_{\mathcal{O}'/\mathcal{O}}$, and $\widehat{\Omega}_{\mathcal{O}'/\mathcal{O}}$ is $p$-completely flat. It then follows that $\widehat{\Omega}_{\mathcal{O}'/\mathcal{O}}$ is $p$-torsion free and flat (cf. Lemma \ref{lemma: val rings flat}).
\end{proof}

\begin{cor}
    Under the assumptions of Theorem \ref{theo: cpt diff mod of val rings}, the module $\widehat{\Omega}_{\mathcal{O}'/\mathcal{O}}$ is classically $\pi$-adically complete.
\end{cor}

\begin{proof}
    By construction, $\widehat{\Omega}_{\mathcal{O}'/\mathcal{O}}$ is derived $\pi$-complete. Since it is flat over $\mathcal{O}'$ by Theorem \ref{theo: cpt diff mod of val rings}, its derived completion agrees with its classical $\pi$-adic completion. In particular, $\widehat{\Omega}_{\mathcal{O}'/\mathcal{O}}$ is classically $\pi$-adically complete.
\end{proof}

The following result provides a useful method for explicitly computing complete differential modules in practice.

\begin{theo}\label{theo: calculate dif module by not complete}
    Let $\mathcal{O}'/\mathcal{O}$ be a (not necessarily complete) valuation ring extension. Then:

\begin{enumerate}[label=(\roman*)]
    \item The $\pi$-adic completion $\widehat{\mathcal{O}}'$ of $\mathcal{O}'$ is a valuation ring;
    \item The complete cotangent complex $\widehat{\mathbb{L}}_{\widehat{\mathcal{O}}'/\mathcal{O}}$ is equal to the derived $\pi$-completion of $\mathbb{L}_{\mathcal{O}'/\mathcal{O}}$.
\end{enumerate}

Moreover, if we further assume that $\mathcal{O}$ is perfectoid, then $\widehat{\Omega}_{\widehat{\mathcal{O}}'/\mathcal{O}}$ is equal to the $\pi$-adic completion of $\Omega_{\mathcal{O}'/\mathcal{O}}$.
\end{theo}

\begin{proof}
    Since $\calO'$ is flat over $\calO$ by Lemma \ref{lemma: val rings flat}, is completion $\widehat{\calO}'$ is equal to the derived $\pi$-completion of $\calO'$. Hence, $\widehat\calO'$ is $\pi$-completely flat over $\calO$ and 
    \[\widehat{\calO}'/\pi\cong \calO'/\pi,\]
    the first claim then follows from the proof of \cite[Lemma 3.13]{Bouis_2023}.

    The second claim is immediate from the definition of derived completion.

For the last claim, since $\mathcal{O}$ is perfectoid, it follows from \cite[Theorem 3.1]{Bouis_2023} that $\mathbb{L}_{\mathcal{O}'/\mathcal{O}} \otimes_{\mathcal{O}}^{\mathbb{L}} \mathcal{O}/\pi$ is a flat $\mathcal{O}'/\pi$-module. Moreover, by \cite[Theorem 6.5.8]{Gabber_2003}, we have $\mathbb{L}_{\mathcal{O}'/\mathcal{O}} = \Omega_{\mathcal{O}'/\mathcal{O}}[0]$. Combining these facts, we deduce that $\Omega_{\mathcal{O}'/\mathcal{O}}$ is flat over $\mathcal{O}'$. By definition, the derived completion of $\Omega_{\mathcal{O}'/\mathcal{O}}$ is therefore isomorphic to its classical $\pi$-adic completion. This completes the proof.
\end{proof}

\noindent\textbf{Uniformization of one dimensional analytic fields: } Let $k$ be an algebraically closed complete non-Archimedean field, and let $\mathcal{K}/k$ be a one-dimensional analytic field (cf. Definition \ref{defi:1 dim an field}).

The following fundamental uniformization result is due to Temkin.

\begin{theo}[Temkin]\label{theo:uniformization}
There exists an element $z \in \mathcal{K} \setminus k$ satisfying the following:
\begin{itemize}
    \item The extension $\mathcal{K}/k(z)$ is unramified. That is, the corresponding extension between rings of integers is finite étale.
\end{itemize}
Such an element $z$ is called a \emph{uniformization} of $\mathcal{K}$.
\end{theo}

\begin{proof}
See \cite[Theorem 6.3.1]{Temkin_2010}.
\end{proof}

\subsection{Commutative algebra in positive characteristic}

\noindent\textbf{General theory:} We explain two types of $F$-singularities that will be used.

\begin{defi}[Split homomorphism, finite homomorphism, $F$-split, $F$-finite]
	Let $f: A \to B$ be a ring homomorphism.
	\begin{enumerate}[label=(\roman*)]
		\item The homomorphism $f$ is called \textit{split} if there exists a left inverse of $f$ as an $A$-module homomorphism.
		\item The homomorphism $f$ is called \textit{finite} if $B$ is a finitely generated $A$-module.
	\end{enumerate}
	Furthermore, for a ring $R$ of characteristic $p$, we say that $R$ is \textit{$F$-split} (resp. \textit{$F$-finite}) if the Frobenius endomorphism $F_R$ of $R$ is split (resp. finite).
\end{defi}

We will also make use of Frobenius liftings of rings of positive characteristic and related constructions.

Let \( R \) be a ring of characteristic \( p \). For any integer \( n \geq 0 \), the differential
\[
d \colon \Omega^n_{R/\mathbb{F}_p} \to \Omega^{n+1}_{R/\mathbb{F}_p}
\]
satisfies \( \dif(a^p \omega) = a^p \dif(\omega) \). In other words, \( \dif \) is an \( R \)-linear map from \( F_* \Omega^n_{R/\mathbb{F}_p} \) to \( F_* \Omega^{n+1}_{R/\mathbb{F}_p} \). Hence, the de Rham complex
\[
0 \to F_* \Omega^0_{R/\dF_p} \to F_* \Omega^1_{R/\dF_p} \to F_* \Omega^2_{R/\dF_{p}} \to \dots
\]
is a complex of \( R \)-modules, and the de Rham cohomology \( H^n_{\mathrm{dR}}(R) \) carries a canonical \( R \)-module structure. The following definition is well-known (see, for example, \cite[(7.2)]{Katz70}).

\begin{defi}\label{defi: inv cartier}
    Let $\Omega_R^n = \Omega_{R/\dF_p}^n$ for any $n\ge 0$.
\begin{enumerate}[label=(\roman*)]
    \item The \emph{inverse Cartier homomorphism}
    \[
    \begin{aligned}
    C_R^{-1}: \Omega_R & \longrightarrow \hol_{\mathrm{dR}}^1(R) \\
    \text{s.t.} \quad C_R^{-1}(f \cdot dg) &:= f^p \cdot g^{p-1} d(g).
    \end{aligned}
    \]
    By taking wedge products, we define the \emph{$n$-th inverse Cartier homomorphism}
    \[C_{R,n}^{-1}:\Omega_R^n\to \hol^n_{\mathrm{dR}}(R)\]

    \item A Cartier lifting is a diagram of $R$-modules:
    \[
    \begin{tikzcd}[row sep=large, column sep=large]
    \Omega_R^1 \arrow[r] \arrow[d, "C^{-1}"'] & 
    (F_*\Omega_R^1)^{d=0} \arrow[ld, "\text{proj}"] \\
    \hol_{\mathrm{dR}}^1(R) &
    \end{tikzcd};
    \]
    \item $R$ is called Cartier liftable if a Cartier lifting exists.
\end{enumerate}

\end{defi}

\begin{defi}
    Keep the notation in Definition \ref{defi: inv cartier}. 
    \begin{enumerate}[label=(\roman*)]
        \item An $F$-lifting of $R$ is a $\dZ/{p^2}$-flat lifting $\widetilde{R}$ of $R$ which is endowed with a Frobenius lifting;
        \item $R$ is called $F$-liftable if an $F$-lifting exists.
    \end{enumerate}
\end{defi}

\begin{lemma}\label{lemma: F lifting-W_2 split-cartier lift}
    If $R$ is $F$-liftable, then $R$ is Cartier liftable.
\end{lemma}

\begin{proof}

Let \( (\widetilde{R}, \varphi) \) be an \( F \)-lifting of \( R \). We define a map 
    \[
        \partial \colon x \in \widetilde{R} \mapsto x^{p-1}\dif x + \dif\left( \frac{\varphi(x) - x^{p}}{p} \right) \in \Omega_{R}.
    \]
    We first verify that \( \partial \) is a derivation from \( \widetilde{R} \) to \( F_{*} \Omega_{R} \). This amounts to checking the following two properties for all \( x, y \in \widetilde{R} \):
    \begin{enumerate}
        \item[(a)] \( \partial(x + y) = \partial(x) + \partial(y) \);
        \item[(b)] \( \partial(xy) = x^{p} \partial(y) + y^{p} \partial(x) \).
    \end{enumerate}

    To prove (a), we compute:
    \[
        \partial(x + y) = (x + y)^{p-1} (\dif x + \dif y) + \dif\left( \frac{\varphi(x) + \varphi(y) - (x + y)^{p}}{p} \right).
    \]
    Observe that
    \[
        \dif\left( \frac{\varphi(x) + \varphi(y) - (x + y)^{p}}{p} \right) = \dif\left( \frac{\varphi(x) + \varphi(y) - x^{p} - y^{p}}{p} \right) - \dif\left( \sum_{j=1}^{p-1} p^{-1} \binom{p}{j} x^{j} y^{p-j} \right).
    \]
    Furthermore, we have
    \[
        \begin{aligned}
            \dif\left( \sum_{j=1}^{p-1} p^{-1} \binom{p}{j} x^{j} y^{p-j} \right) &= \sum_{j=1}^{p-1} p^{-1} \binom{p}{j} \left( j x^{j-1} y^{p-j} \dif x + (p - j) x^{j} y^{p-j-1} \dif y \right) \\
            &= \sum_{j=0}^{p-2} \binom{p-1}{j} x^{j} y^{p-1-j} \dif x + \sum_{j=1}^{p-1} \binom{p-1}{j} x^{j} y^{p-1-j} \dif y \\
            &= \sum_{j=0}^{p-1} \binom{p-1}{j} x^{j} y^{p-1-j} (\dif x + \dif y) - x^{p-1} \dif x - y^{p-1} \dif y \\
            &= (x + y)^{p-1} (\dif x + \dif y) - x^{p-1} \dif x - y^{p-1} \dif y.
        \end{aligned}
    \]
    Using this identity, part (a) follows directly.

    To prove (b), define the map \( \delta \colon \widetilde{R} \to R \) by \( \delta(x) = \frac{\varphi(x) - x^{p}}{p} \). Then we have
    \[
        \delta(xy) = x^{p} \delta(y) + y^{p} \delta(x),
    \]
    which implies
    \[
        \dif(\delta(xy)) = x^{p} \dif(\delta(y)) + y^{p} \dif(\delta(x)).
    \]
    This establishes (b) and completes the proof that \( \partial \) is a derivation.
\end{proof}

\begin{rmk}
    In the proof above, when \( R \) is \emph{cotangent-smooth} (that is, when the cotangent complex \( \mathbb{L}_{R/\mathbb{F}_p} \) is a flat module), the operator \( \partial \) coincides with \( \frac{d\varphi}{p} \).
\end{rmk}

\noindent\textbf{Specialize to valuation rings in positive characteristic:} All concepts introduced above will be applied exclusively when \( R \) is a valuation ring. Moreover, we will generally work in the complete setting.

From now on, let \( \mathcal{O} \) be a Tate valuation ring over \( \mathbb{F}_p \), and let \( \varpi \in \mathcal{O} \) be a pseudo-uniformizer. We recall the following fundamental result:
\begin{theo}[Gabber]\label{theo:flat cotangent cpx of val ring in char p and cartier smooth}
    The cotangent complex \( \mathbb{L}_{\mathcal{O}/\mathbb{F}_p} \) is isomorphic to a flat \( \mathcal{O} \)-module concentrated in degree~0. Moreover, the inverse Cartier homomorphism
\[
C^{-1}_{\mathcal{O}, n}:\Omega^n_{\calO/\dF_p}\to \hol^n_{\mathrm{dR}}(\calO)
\]
is an isomorphism for all \( n \).
\end{theo}

\begin{proof}
    See the appendix of \cite{Kerz_2021}.
\end{proof}

Define $\widehat{\Omega}_{\calO}$ as the $\varpi$-adic completion of the module of differentials $\Omega_{\calO/\dF_p}$. By Theorem \ref{theo:flat cotangent cpx of val ring in char p and cartier smooth}, \( \widehat{\Omega}_{\mathcal{O}/\mathbb{F}_p} \) is a \( \varpi \)-completely flat \( \mathcal{O} \)-module, and it agrees with the \( \varpi \)-adic derived completion of \( \mathbb{L}_{\mathcal{O}/\mathbb{F}_p} \). We will also make use of the \emph{complete de Rham cohomology} of \( \mathcal{O} \).
\begin{defi}
    Make the following definitions:
    \begin{enumerate}[label=(\roman*)]
        \item For any \( n \geq 0 \), define \( \widehat{\Omega}^n_{\mathcal{O}} \) as the \( \varpi \)-adic completion of \( \Omega^n_{\mathcal{O}/\mathbb{F}_p} \);
        \item Define \( \widehat{\mathrm{dR}}_{\mathcal{O}} \) as the \( \varpi \)-adic \emph{derived} completion of the de Rham complex \( \mathrm{dR}_{\mathcal{O}/\mathbb{F}_p} \);
        \item For any integer \( n \), define \( \widehat{\hol}^n_{\mathrm{dR}}(\mathcal{O}) \) as the \( n \)-th cohomology group of \( \widehat{\mathrm{dR}}_{\mathcal{O}} \), which carries a natural \( \mathcal{O} \)-module structure.
    \end{enumerate}
\end{defi}

In fact, \( \widehat{\mathrm{dR}}_{\mathcal{O}} \) is quasi-isomorphic to its `naive' completion, obtained by taking the \( \varpi \)-adic completion termwise.

\begin{lemma}\label{lemma: calculate complete de rham cpx}
    The complex \( \widehat{\mathrm{dR}}_{\mathcal{O}} \) is canonically isomorphic to the complex
\[
  \mathrm{dR}_{\calO}^\wedge:=[0 \to F_*\mathcal{O} \to F_*\widehat{\Omega}_{\mathcal{O}} \to F_*\widehat{\Omega}^2_{\mathcal{O}} \to \cdots],
\]
obtained by applying termwise \( \varpi \)-adic completion to the de Rham complex.
\end{lemma}

\begin{rmk}
    Note that for any $\calO$-module $M$, the Frobenius twist of its $\varpi$-adic completion $F_*\widehat{M}$ is canonically isomorphic to the completion of it Frobenius twist $\widehat{(F_*M)}$.
\end{rmk}

\begin{proof}
    Since each term of ${\mathrm{dR}}^\wedge_{\mathcal{O}} $ is $\varpi$-adically complete, it is derived $\varpi$-adically complete.

    Consider the canonical morphism \( \mathrm{dR}_{\mathcal{O}} \to {\mathrm{dR}}^\wedge_{\mathcal{O}} \). Since each \( \Omega^n_{\mathcal{O}} \) is flat over \( \mathcal{O} \), so is its completion (by Corollary \ref{cor: val ring cpt of flat is flat}) Hence, for any \( n \geq 1 \), we have
\begin{align*}
    \mathcal{O}/\varpi^n \otimes^{\mathbf{L}}_{\mathcal{O}} {\mathrm{dR}}^\wedge_{\mathcal{O}} 
&= \left[ 0 \to F_*(\mathcal{O}/\varpi^{pn}) \to \mathcal{O}/\varpi^n \otimes_{\mathcal{O}} F_*\Omega_{\mathcal{O}} \to \mathcal{O}/\varpi^n \otimes_{\mathcal{O}} F_*\Omega^2_{\mathcal{O}} \to \cdots \right]\\
&=\mathcal{O}/\varpi^n \otimes^{\mathbf{L}}_{\mathcal{O}} {\mathrm{dR}}_{\mathcal{O}}.
\end{align*}

Passing to the derived limit, we obtain the desired isomorphism.
\end{proof}

\begin{lemma}\label{lemma: calculate complete de rham cohomology}
    The complete de Rham cohomology \( \widehat{\hol}_{\mathrm{dR}}^{n}(\mathcal{O}) \) is canonically isomorphic to the classical \( \varpi \)-adic completion of the naive de Rham cohomology \( \hol_{\mathrm{dR}}^{n}(\mathcal{O}) \).
\end{lemma}

\begin{proof}
    For any integer \( n \geq 0 \), define \( B_{n} \) as the image of the differential \( \dif \colon F_{*}\Omega_{\mathcal{O}}^{n-1} \to F_{*}\Omega_{\mathcal{O}}^{n} \) (with the convention that \( \Omega_{\mathcal{O}}^{-1} = 0 \)), and let \( Z_{n} \) be the kernel of the differential \( \dif \colon F_{*}\Omega_{\mathcal{O}}^{n} \to F_{*}\Omega_{\mathcal{O}}^{n+1} \). By definition, we have exact sequences:
\begin{align}
0 \to Z_{n} \to B_{n} \to \hol_{\mathrm{dR}}^{n}(\mathcal{O}) \to 0, \label{eq: exact in complete dr 1} \\
0 \to B_{n} \to F_{*}\Omega_{\mathcal{O}}^{n} \to Z_{n+1} \to 0. \label{eq: exact in complete dr 2}
\end{align}

Since \( B_{n} \) and \( Z_{n} \) are submodules of \( F_{*}\Omega_{\mathcal{O}}^{n} \), they are \( \varpi \)-torsion free and hence flat by Lemma \ref{lemma: val rings flat}. By Theorem \ref{theo:flat cotangent cpx of val ring in char p and cartier smooth}, the inverse Cartier homomorphism is an isomorphism; in particular, \( H_{\mathrm{dR}}^{n}(\mathcal{O}) \) is flat for all \( n \). These flatness conditions imply that the exact sequences (\ref{eq: exact in complete dr 1}) and (\ref{eq: exact in complete dr 2}) remain exact after \( \varpi \)-adic completion. Consequently, the \( \varpi \)-adic completion of \( H_{\mathrm{dR}}^{n}(\mathcal{O}) \) equals the \( n \)-th cohomology of the complex
\[
0 \to F_{*}\mathcal{O} \to F_{*}\widehat{\Omega}_{\mathcal{O}} \to F_{*}\widehat{\Omega}_{\mathcal{O}}^{2} \to \cdots.
\]
The result now follows from Lemma \ref{lemma: calculate complete de rham cpx}.
\end{proof}

\subsection{$\delta$-liftings}
In this article, we will freely use the theory of prisms introduced in \cite{Bhatt_2022}.

\begin{defi}
    Let $(R,I)$ be a bounded prism in the sense of \cite[Definition 3.2]{Bhatt_2022}, and set $\overline R=R/I$. For an $\overline{R}$-algebra $S$, a \emph{$\delta$-lifting of $S$ over $R$} is a prism $\widetilde{S}$ over $R$ equipped with an isomorphism \( \widetilde{S}/I\widetilde{S} \cong S \). Such a \(\delta\)-lifting \( \widetilde{S} \) is called \emph{flat} if \( \widetilde{S} \) is \( (p, I) \)-completely flat over \( R \).
\end{defi}

\subsection{Berkovich $\dA^1$}

Let \( k \) be an algebraically closed complete non-Archimedean field. The points of the Berkovich analytification of \( \mathbb{A}^1_k \) (denoted \( \mathbb{A}^{1,\mathrm{an}}_k \)) are, by definition, the set of all seminorms on \( k[t] \) that extend the given norm on \( k \). Let \( |\bigcdot| \) denote the non-Archimedean norm on \( k \).

\begin{theo}
    Every point of \( \mathbb{A}^{1,\mathrm{an}}_k \) belongs to exactly one of the following four types:
    \begin{enumerate}[label=(\roman*)]
        \item \textbf{Type 1 points:} For a \( k \)-algebra homomorphism \( m \colon k[t] \to k \), the seminorm \( f \mapsto |m(f)| \).
        \item \textbf{Type 2 points:} For \( r \in |k^\times| \) and \( x \in k \), the Gauss norm with center \( x \) and radius \( r \).
        \item \textbf{Type 3 points:} For \( r \notin |k^\times| \) and \( x \in k \), the Gauss norm with center \( x \) and radius \( r \).
        \item \textbf{Type 4 points:} For a nested sequence of closed balls \( B_1 \supseteq B_2 \supseteq \cdots \) in \( k \) with empty intersection, the seminorm \( f \mapsto \lim_{i} \|f\|_i \), where \( \|\bigcdot\|_i \) denotes the Gauss norm on the ball \( B_i \). \footnote{This limit must exist.}
    \end{enumerate}
\end{theo}

\begin{proof}
    See \cite[1.4.3, 1.4.4, 4.1]{Berkovich_2012}
\end{proof}

Note that a point \( v \in \mathbb{A}^{1,\mathrm{an}}_k \) has zero support unless it is of type 1. Furthermore, the value group and residue field can be explicitly described for each type of points.

\begin{theo}\label{theo: type coincide}
    Let $v$ be a point in $\dA^{1,\an}_k$ that is not of type 1.
    \begin{enumerate}[label=(\roman*)]
    \item If \( v \) is of type~2, then its residue field is a proper extension of the residue field of \( k \).
    \item If \( v \) is of type~3, then its value group is a proper extension of the value group of \( k \).
    \item If \( v \) is of type~4, then at least one of the above two properties holds.
\end{enumerate}
\end{theo}

\begin{proof}
    See \cite[1.4.4]{Berkovich_2012}
\end{proof}

In particular, Berkovich's classification coincides with Definition \ref{defi: Berkovich classify 1 dim geo rings}.

\section{Proof of the main theorems}

The goal of this section is to prove Theorem \ref{main theo in intro} and Theorem \ref{theo: main 2 f finite}.

\subsection{From $F$-lifting to $F$-splitting}

In this part, we fix a rank one Tate valuation ring $\calO$ over $\dF_p$ with a pseudo-uniformizer $\varpi\in \calO$. Denote by $\calK$ the fractional field of $\calO$. Our goal is to prove the following:

\begin{theo}\label{theo: f lift to f split}
    Assume that the vector space $\widehat{\Omega}_{\calO}[\frac{1}{\varpi}]$ over $\calK$ is finite-dimensional. If $\calO$ is $F$-liftable, then $\calO$ is $F$-split.
\end{theo}

Before proving the theorem, we establish the following key lemma.

\begin{lemma}\label{lemma: end of 1 rank almost free mod}
    Let $\mathcal{V}$ be a $\varpi$-adically complete flat $\mathcal{O}$-module such that $\dim_{\mathcal{K}} \mathcal{V}[\frac{1}{\varpi}] = 1$. Then every $\mathcal{O}$-endomorphism of $\mathcal{V}$ is given by scalar multiplication.
\end{lemma}

\begin{proof}
    Since $\mathcal{V}$ is flat over $\mathcal{O}$, the localization map $\mathcal{V} \to \mathcal{V}[\frac{1}{\varpi}]$ is injective. As $\dim_{\mathcal{K}} \mathcal{V}[\frac{1}{\varpi}] = 1$, we have
\[
\mathcal{K} = \operatorname{End}_{\mathcal{K}}(\mathcal{V}[\tfrac{1}{\varpi}]) = \operatorname{End}_{\mathcal{O}}(\mathcal{V}[\tfrac{1}{\varpi}]).
\]
It follows that
\[
\operatorname{End}_{\mathcal{O}}(\mathcal{V}) = \{ a \in \mathcal{K} : a \mathcal{V} \subseteq \mathcal{V}\text{ as submodules of $\calV[\frac{1}{\varpi}]$} \}.
\]
Now, since $\mathcal{V}$ is $\varpi$-adically complete, there exists an element $a \in \mathcal{K}$ such that $a \mathcal{V} \not\subseteq \mathcal{V}$. Consequently, $\operatorname{End}_{\mathcal{O}}(\mathcal{V})$ is a proper subring of $\mathcal{K}$ containing $\mathcal{O}$. Therefore, we conclude that $\operatorname{End}_{\mathcal{O}}(\mathcal{V}) = \mathcal{O}$.
\end{proof}

\begin{proof}[Proof of Theorem \ref{theo: f lift to f split}]

By Lemma \ref{lemma: F lifting-W_2 split-cartier lift}, there exists a Cartier lifting $\Phi: \Omega_{\mathcal{O}/\mathbb{F}_p} \to F_*\Omega_{\mathcal{O}/\mathbb{F}_p}$.

Let $n = \dim_{\mathcal{K}} \widehat{\Omega}_{\mathcal{O}}[\frac{1}{\varpi}]$. Denote by $\widehat{\Phi}_n: \widehat{\Omega}_{\mathcal{O}}^n \to F_*\widehat{\Omega}_{\mathcal{O}}^n$ the $n$-th complete wedge product of $\Phi$. By Lemma \ref{lemma: calculate complete de rham cpx} and Lemma \ref{lemma: calculate complete de rham cohomology}, there exists a canonical projection $\operatorname{pr}_n$ from $F_*\widehat{\Omega}_{\mathcal{O}}^n$ to $\widehat{H}_{\mathrm{dR}}^n(\mathcal{O})$. We claim that the diagram
\begin{equation}\label{eq: diagram in F-lift to F-split}
        \begin{tikzcd}[column sep=large]
        \widehat\Omega^n_{\calO} \ar[dr,"\widehat C_{n}^{-1}"swap]\arrow{r}{\widehat\Phi_n} & F_*\widehat{\Omega}_{\calO}^n\arrow{d}{\mathrm{pr}_n}\\
        {} & \widehat{\hol}^n_{\mathrm{dR}}(\calO)
    \end{tikzcd}
    \end{equation}
commutes, where $\widehat{C}_n^{-1}$ is the completion of the $n$-th inverse Cartier homomorphism.

To establish this commutativity, it suffices to verify the corresponding statement in the incomplete case. Specifically, we need to show that the following diagram commutes:
\begin{equation}\label{eq: diagram in F-lift to F-split imcomplete}
        \begin{tikzcd}[column sep=large]
        \Omega^n_{\calO} \ar[dr," C_{n}^{-1}"swap]\arrow{r}{\Phi_n} & \big(F_*{\Omega}^{n}_{\calO}\big)^{\dif=0}\arrow{d}{\mathrm{pr}_n}\\
        {} & {\hol}^n_{\mathrm{dR}}(\calO)
    \end{tikzcd}.
\end{equation}

By the functoriality of cup products, it is sufficient to verify the case $n=1$, which has already been established in Lemma \ref{lemma: F lifting-W_2 split-cartier lift}.

We now define the key construction. Denote by $\widehat{C}_n$ the inverse of $\widehat{C}_n^{-1}$, and for any $f \in \mathcal{O}$, denote by $\operatorname{mult}_f$ the multiplication-by-$f$ map on $\widehat{\Omega}_{\mathcal{O}}^n$. When viewing $\operatorname{mult}_f$ as an endomorphism of $F_*\widehat{\Omega}_{\mathcal{O}}^n$, we observe that $\operatorname{mult}_{g^p f} = g \cdot \operatorname{mult}_f$ for any $f, g \in \mathcal{O}$. 

By Lemma \ref{lemma: end of 1 rank almost free mod}, for each $f \in \mathcal{O}$, there exists a unique $\eta(f) \in \mathcal{O}$ such that for any $x \in \widehat{\Omega}_{\mathcal{O}}^n$,
\[
\eta(f)x = (\widehat{C}_n \circ \operatorname{pr}_n \circ \operatorname{mult}_f \circ \widehat{\Phi}_n)(x).
\]

We now verify that the map $\eta$ defines an $F$-splitting. First, $\eta(1) = 1$ follows from the commutativity of Diagram (\ref{eq: diagram in F-lift to F-split}). Moreover, for any $f, g \in \mathcal{O}$, the identity $\operatorname{mult}_{g^p f} = g \cdot \operatorname{mult}_f$ implies that $\eta(g^p f) = g \eta(f)$, which completes the proof.
\end{proof}

\begin{rmk}
     The technique employed in this proof is adapted from the argument showing that for a smooth variety $X$ over a perfect field $k$, if $X$ admits an $F$-lifting over $W_2(k)$, then the structure sheaf $\mathcal{O}_X$ is $F$-split. We refer to \cite{Buch_1997} for further details.
\end{rmk}

In our applications, we will use another form of Theorem \ref{theo: f lift to f split}. For simplicity, assume that the value group of $\mathcal{K}$ is dense in $\mathbb{R}_{>0}$. Suppose $\mathcal{O}^{++}$ is the maximal ideal of $\mathcal{O}$, and we discuss almost mathematics with respect to this ideal (cf. \cite[Section 2]{scholze2012padic}).

Assume that $M$ is a flat, $\varpi$-adically complete module over $\mathcal{O}$.

\begin{theo}\label{theo: equi cond of finite dim}
	The following are equivalent:
	\begin{enumerate}[label=(\roman*)]
		\item $M$ is almost finitely generated;
		\item $M/\varpi$ is an almost finitely generated $\mathcal{O}/\varpi$-module;
		\item $M[\frac{1}{\varpi}]$ is a finite-dimensional $\mathcal{K}$-vector space.	
	\end{enumerate}
\end{theo}

\begin{proof}
	(i) $\Rightarrow$ (ii) is obvious.
	
	For (ii) $\Rightarrow$ (iii): Since the value group of $\mathcal{K}$ is dense, we can choose $\eta\in \mathcal{O}^{++}$ such that $\varpi\in\eta^2\mathcal{O}$. By the definition of almost finitely generated, we can choose elements $m_1,m_2,\dots,m_r\in M$ such that the submodule they generate in $M/\varpi$ contains $\eta M/\varpi M$. Let $N$ be the submodule of $M$ generated by $m_1,m_2,\dots,m_r$, then we have
	\begin{equation}\label{eq: N+varpi M supseteq eta M}N+\eta^2 M\supseteq N+\varpi M\supseteq \eta M.\end{equation}
	We will show that $N\supseteq \eta M$. For any $m\in M$, we construct by induction sequences $\{a_{i,j}:i\geq 1,j=1,2,\dots,r\}$ satisfying:
	\begin{enumerate}[label=(\alph*)]
		\item $a_{i,j}-a_{i+1,j}\in \eta^i M$;
		\item $\eta m-\sum_{j=1}^r a_{i,j}m_j\in \eta^{i+1} M$.
	\end{enumerate}
	For $i=1$, from (\ref{eq: N+varpi M supseteq eta M}), there exist $\{a_{1,j}:j=1,2,\dots, r\}$ such that 
	\[\eta m-\sum_{j=1}^r a_{1,j}m_j\in \eta^2 M.\] 
	Suppose we have constructed $\{a_{n,j}:j=1,2,\dots,r\}$. Let $\eta^{n+1}z=\eta m-\sum_{j=1}^r a_{n,j}m_j$. Then by (\ref{eq: N+varpi M supseteq eta M}), there exist $b_1,b_2,\dots,b_r\in\mathcal{O}$ such that $\eta z-\sum_{j=1}^rb_jm_j\in\eta^2 M$. Let $a_{n+1,j}=a_{n,j}+\eta^n b_j$, then
	\begin{align*}
	\eta m-\sum_{j=1}^r a_{n+1,j}m_j&=\eta^{n+1}z-\eta^n\sum_{j=1}^rb_jm_j\\
	&=\eta^n(\eta z-\sum_{j=1}^r b_jm_j)\in \eta^{n+1}M.
	\end{align*}
	This completes the construction. Using this, we deduce by an approximation argument that $N\supseteq \eta M$. Thus, $M[\frac{1}{\varpi}]=M[\frac{1}{\eta}]=N[\frac{1}{\eta}]$ is finite dimensional.
	
	For (iii) $\Rightarrow$ (i): We proceed by induction on the dimension. When $\dim_{\mathcal{K}}(M[\frac{1}{\varpi}])=1$, we can regard $M$ as a bounded open submodule of $\mathcal{K}$. Therefore, the norms of elements in $M$ have a finite supremum $r$, and it is not difficult to verify that $M$ is almost isomorphic to the module $\{x\in \mathcal{K}: |x|\leq r\}$. Hence $M$ is almost finitely generated. Now assume the statement holds for $\dim_{\mathcal{K}}(M[\frac{1}{\varpi}])=d$, and suppose $\dim_{\mathcal{K}}(M[\frac{1}{\varpi}])=d+1$. Take a $d$-dimensional subspace $N'$ of $M[\frac{1}{\varpi}]$ and let $N=N'\cap M$. It is not difficult to verify that $N$ is a closed submodule of $M$, hence also $\varpi$-adically complete. By the induction hypothesis, $N$ is almost finitely generated. Moreover, by construction $M/N$ is torsion-free and it is not difficult to verify that it is also $\varpi$-adically complete. Therefore, by the 1-dimensional case, $M/N$ is almost finitely generated. Then by \cite[Lemma 2.3.18]{Gabber_2003}, it follows that $M$ is almost finitely generated.
\end{proof}

\subsection{A key theorem}

We now present a fundamental result that provides a sufficient condition that a valuation ring is not F-split.

\begin{theo}\label{theo: non f split}
    Let $\mathcal{K}$ be a complete non-Archimedean field of characteristic $p$. Assume the following conditions hold:
    \begin{enumerate}[label=(\roman*)]
        \item $\calK^+$ is not perfect;
        \item The value group of $\calK$ is $p$-divisible;
        \item The residue field is perfect.
    \end{enumerate}
    Then $\calK^+$ is not $F$-split.
\end{theo}

\begin{proof}
    Assume, by way of contradiction, that there exists a $\mathcal{K}^{+p}$-linear splitting $\sigma \colon \mathcal{K}^+ \to \mathcal{K}^{+p}$. Let $|\bigcdot|$ denote the norm on $\mathcal{K}$. For any $f \in \mathcal{K}^+ \setminus \mathcal{K}^{+p}$, define
\[
r_f := \inf \{ |f - g^p| : g \in \mathcal{K} \}.
\]
Since $\mathcal{K}^p$ is closed in $\mathcal{K}$, we have $r_f > 0$. Moreover, it follows directly from the definition that $r_f \leq |f|$.

We first claim that there does not exist any $g \in \mathcal{K}$ such that $|f - g^p| = r_f$. Suppose, to the contrary, that such a $g$ exists. By replacing $f$ with $f - g^p$, we may reduce to the case where $g = 0$ and $r_f = |f|$. Furthermore, since the value group of $\mathcal{K}$ is $p$-divisible, we can choose $h \in \mathcal{K}$ such that $|f| = |h^p|$. Replacing $f$ by $f/h^p$, we may further assume that $|f| = r_f = 1$. However, since $\mathcal{K}$ has a perfect residue field, there exists some $g \in \mathcal{K}^+$ such that $|f - g^p| < 1$, which contradicts the assumption that $r_f = 1$. This establishes our claim.

As an immediate corollary, we have $r_f < |f - 0^p| = |f|$. Consequently, we can choose a sequence of elements $g_1, g_2, \dotsc \in \mathcal{K}^+$ such that:
\begin{enumerate}[label=(\Roman*)]
    \item the quantities $r_i := |f - g_i^p|$ satisfy $r_i > r_{i+1}$ for all $i$;
    \item $\lim_{i \to \infty} r_i = r_f$.
\end{enumerate}

For each index $i$, choose $\lambda_i \in \mathcal{K}^+$ such that $|\lambda_i^p| = r_i$. Then we have
\[
\frac{f - g_i^p}{\lambda_i^p} \in \mathcal{K}^+.
\]
It follows that $|\sigma(f - g_i^p)| \leq r_i$ for all $i$. Thus, we obtain the estimate
\[
|\sigma(f) - f| \leq \max \{ |f - g_i^p|, |\sigma(f) - g_i^p| \} = r_i
\]
for all $i$. Taking the limit as $i \to \infty$, we conclude that $|\sigma(f) - f| \leq r_f$. On the other hand, since $\sigma(f)\in\calK^p$, $|f-\sigma(f)|>r_f$. This contradiction completes the proof.
\end{proof}

\begin{rmk}
    We note that the above theorem has already been established in \cite[Section 5, (14)]{Datta_2021}. However, the proof presented here differs slightly from the one given in that reference.
\end{rmk}

\subsection{Proof of Theorem \ref{main theo in intro}}\label{pr to main 1}

In this subsection, let \( k \) be an algebraically closed complete non-Archimedean field whose residue field has characteristic \( p \). We denote by \( k^\flat \) the tilt of \( k \), by \( \mathfrak{o} \) the ring of integers of \( k \), and by \( \mathfrak{m} \) the maximal ideal of \( \mathfrak{o} \). Fix a nonzero element \( \varpi \in \mathfrak{m} \), and fix an element \( \varpi^\flat \in k^\flat \) whose untilt is \( \varpi \). Let \( A_{\inf} = A_{\inf}(\mathfrak{o}) \) be Fontaine's period ring, and let \( \theta: A_{\inf} \to \mathfrak{o} \) be Fontaine's map.

Let \( \mathcal{K}/k \) be a one-dimensional analytic field (cf. Definition \ref{defi:1 dim an field}).

\begin{theo}\label{theo: type 23 has delta lifting}
	Assume that $\mathcal{K}$ is Abhyankar (recall that this is equivalent to $\calK$ being type 2 or type 3). Then $\mathcal{K}^+$ has a flat $\delta$-lift over $\Ainf$.
\end{theo}

\begin{proof}
	By the Theorem \ref{theo:uniformization}, choose a uniformization $t\in \calK$. Let $\calE_0=k(t)$ and $\calE$ be the closure of $\calE_0$ in $\calK$.
    
	Since the embedding $\mathcal{E}^+ \to \mathcal{K}^+$ is finite \'etale, basic deformation theory implies that if $\mathcal{E}^+$ admits the required $\delta$-lift, then so does $\mathcal{K}^+$. Moreover, by \cite[Proposition 2.3.3]{Datta_2016} if $\mathcal{K}$ is of type 2 (resp. type 3), then $\mathcal{E}$ is also of type 2 (resp. type 3). We now observe that the norm on $\mathcal{E}_0$ restricted from $\mathcal{K}$ (denoted by $\|\bigcdot\|$) is a Gauss norm. Without loss of generality, we may assume that $0$ is a center of $\|\bigcdot\|$ and that $r$ denotes its radius.

    We now divide the argument into two cases, depending on the type of $\calE$.
	
	If $r \in |k^\times|$ ($\calE$ is of type 2), assume that $a \in k$ such that $|a| = r$.     
    Via the transformation $t \mapsto t/a$, we may normalize so that $a = 1$. Then $\mathcal{E}_0^+$ is equal to the localization of $\mathfrak{o}[t]$ at the prime ideal $\mathfrak{m} \mathfrak{o}[t]$. Now consider the polynomial algebra $\Ainf[t]$ and its ideal $P = \theta[t]^{-1}(\frakm\frako[t])$, where $$\theta[t]: \Ainf[t] \to \mathfrak{o}[t]$$ is the homomorphism obtained by extending $\theta$ and sending $t$ to $t$. 
    At this point, take $\widetilde{\mathcal{E}}^+$ to be the $(p,[\varpi^{\flat}])$-adic completion of the localization $\Ainf[t]_P$. Then it is a $p$-torsion free ring, and define a $\delta$-structure on it by extending the $\delta$-structure on $\Ainf$ via $\delta(t)=0$ (cf. \cite[Lemma 2.15]{Bhatt_2022}).

    If $r \notin |k^\times|$ ($\calE$ is of type 3), then take two sequences $\{z_i: i \ge 1\}$ and $\{w_i: i \ge 1\}$ in $k^\flat$ satisfying:
    \begin{itemize}
    \item $|z_i^\sharp| > |z_{i+1}^\sharp| > r > |w_{i+1}^\sharp| > |w_i^\sharp|$;
    \item $\lim_{i\to \infty}|z^\sharp_i| = \lim_{i\to \infty}|w^\sharp_i| = r$.
    \end{itemize}
    For each $i$, let
\[
R_i = \mathfrak{o}[x_i, y_i]/(x_iy_i - w_i^\sharp {z_i^{\sharp\ -1}})
\]
and let $\mathfrak{p}_i$ be the prime ideal generated by $x_i, y_i, \mathfrak{m}$. Then $(R_i)_{\mathfrak{p}_i}$ can be regarded as a subring of $\mathcal{E}^+_0$ via the map sending $x_i$ to $\frac{w_i^\sharp}{t}$. For any $i$, define an $\mathfrak{o}$-algebra homomorphism from $(R_{i})_{\mathfrak{p}_i}$ to $(R_{i+1})_{\mathfrak{p}_{i+1}}$ by
\[
x_i \mapsto w_i^{\sharp}{w_{i+1}^{\sharp\ -1}}x_{i+1};\quad y_i \mapsto {z_i^{\sharp\ -1}}z_{i+1}^\sharp y_{i+1}.
\]
We have (see Appendix \ref{appd:A})
\[
\mathcal{E}_0^+ = \varinjlim_{i}\ (R_i)_{\mathfrak{p}_i}.
\]

For any $i$, define
\[
\widetilde{R}_i = \Ainf[x_i, y_i] / (x_i y_i - [w_i z_i^{-1}])
\]
and let $\widetilde{\mathfrak{p}}_i$ be the prime ideal of $\widetilde{R}_i$ generated by $x_i$, $y_i$, and the maximal ideal of $\Ainf$. Define an $\Ainf$-algebra homomorphism $\widetilde{R}_i \to \widetilde{R}_{i+1}$ by sending $x_i$ to $[w_i][w_{i+1}]^{-1} x_{i+1}$ and $y_i$ to $[z_{i+1}][z_i]^{-1} y_{i+1}$. Let
\[
\widetilde{\mathcal{E}}^+ = \left( \varinjlim_i\ (\widetilde{R}_i)_{\widetilde{\mathfrak{p}}_i} \right)^\wedge
\]
where the completion is the $(p, [\varpi^\flat])$-adic completion.

We claim that there exists a $\delta$-structure on $\widetilde{\calE}^+$ extending the $\delta$-structure on $\Ainf$ such that $\delta(x_i)=\delta(y_i)=0$ for all $i$. Indeed, since $\widetilde{R}_i$ is free as an $\Ainf$-module, we have that
\[\widetilde{\calE}^+_0:=\varinjlim_i (\widetilde{R}_i)_{\widetilde{\mathfrak{p}}_i}\]
is flat over $\Ainf$. In particular, $\widetilde{\calE}_0^+$ is $p$-torsion free. Hence, to give a $\delta$-structure on $\widetilde\calE_0^+$, it suffices to give a Frobenius lifting on $\widetilde{\calE}^+_0$. It is easy to check that there exists Frobenius lifting on $\widetilde{\calE}_0^+$ extending the Witt vector Frobenius on $\Ainf$ and sending each $x_i$ to $x_i^p$, $y_i$ to $y_i^p$. Hence, there exists a $\delta$-structure on $\widetilde{\calE}_0^+$ extending the $\delta$-structure on $\Ainf$ such that $\delta(x_i)=\delta(y_i)=0$ for all $x_i$ and $y_i$. By \cite[Lemma 2.17]{Bhatt_2022}, we can extend this $\delta$-structure to $\widetilde{\calE}^+$.
\end{proof}

\begin{lemma}\label{lemma: dim of dif type 4}
    Let $\calK$ be of type 4. Then, the linear space $\widehat{\Omega}_{\calK^+/\frako}[\frac{1}{\varpi}]$ over $\calK$ is of dimension $1$.
\end{lemma}

\begin{proof}
    Choose a uniformization $z\in\calK$. As $\overline{k(z)}^+\subseteq \calK^+$ is finite \'etale, it suffices to prove that $\widehat{\Omega}_{\overline{k(z)}^+/\frako}[\frac{1}\varpi]$ is of dimension $1$ over $\overline{k(z)}$. Denote by $\calE_0$ the field $k(z)\subseteq \calK$ with the restriction norm and $\widehat\calE_0$ be its completion. Obviously, $\overline{k(z)}$ is identified with $\widehat\calE_0$

    According to \cite[Proposition 2.3.3]{Datta_2016}, the completion $\widehat{\mathcal{E}}_0$ is of type 4. This allows us to construct two sequences of elements: $\{x_i : i = 1, 2, \dots\} \subset k$ and $\{z_i : i = 1, 2, \dots\}$ satisfying the following conditions:
\begin{enumerate}[label=(\alph*)]
    \item $|x_i - x_{i+1}| \leq |z_i|$ and $|z_i| > |z_{i+1}|$;
    \item $\bigcap_{i \geq 1} B_k(x_i, |z_i|) = \emptyset$;
    \item The norm $\|\cdot\|$ on $\mathcal{E}_0$ equals the limit of the Gauss norms $\|\bigcdot\|_i$ with centers $x_i$ and radii $|z_i|$.
\end{enumerate}
    Define $R_i$ (resp. $R_i^+$) as the Tate algebra
    \[R_i=k\langle\frac{t-x_i}{z_i}\rangle\ \text{ (resp. $\frako\langle\frac{t-x_i}{z_i}\rangle$)}.\]
    By definition, $\|\frac{t-x_i}{z_i}\|\leq 1$, thus there exists a canonical continuous homomorphism $R_i\to \widehat{\calK}_0$ sending $t$ to $t$ and $R_i^+$ to $\widehat\calE_0^+$. Hence, these maps induce a canonical homomorphism $\rho$ (resp. $\rho^+$) from $R_\infty:=\varinjlim_i\ R_i$ (resp. $R_\infty^+:=\varinjlim_i\ R_i^+$) to $\widehat{\calK}_0$ (resp. $\widehat{\calK}_0^+$). By abuse of notation, for any $i\leq i'$ and $f\in R_i$, we denote by $\|f\|_{i'}$ the Gauss norm of $f$ with center $x_{i'}$ and radius $|z_{i'}|$. Furthermore, for any $i$ and $f\in R_i$, denote $\|\rho(f)\|$ by $\|f\|$.

    Since every quotient of $R_i$ by a nonzero prime ideal is isomorphic to $k$, the homomorphism $\rho$ is injective. Consequently, $\rho^+$ is also injective. We now establish the following key properties:
\begin{enumerate}[label=(\alph*)]
    \item $R_\infty$ is a field;
    \item $R_\infty^+ = \{f \in R_\infty : \|f\| \leq 1\}$.
\end{enumerate}

    To prove (a), let $0 \neq f \in R_i$. Since $\bigcap_{j} B_k(x_j, |z_j|) = \emptyset$, we can choose an index $i' > i$ such that the ball $B_k(x_{i'}, |z_{i'}|)$ contains no zeros of $f$. This implies that $f$ is invertible in $R_{i'}$, which establishes that $R_\infty$ is a field.

    For (b), let $f \in R_i$. We claim that there exists an index $i' \geq i$ such that $\|f\| = \|f\|_{i'}$. To establish this claim, we will need the following lemma.

    \begin{lemma}\label{lemma: lemma in the lemma of proving dif dim=1}
        For any index $i$ and any nonzero element $f \in R_i$, there exists an integer $N \geq 1$ such that for all $i' \geq i$,
\[
\|f\|_{i'} > \|\varpi^N\|_{i'} = |\varpi^N|.
\]
    \end{lemma}
    \begin{proof}
        We proceed by contradiction. Suppose, to the contrary, that for every positive integer $N$, there exists some index $i' \geq i$ such that $\|f\|_{i'} \leq |\varpi^N|$. This would imply that $f \in \varpi^N R_{i'}^+$ for every $N$, and consequently $\rho(f) \in \varpi^N \widehat{\mathcal{E}}_0^+$ for all $N$. It would then follow that $\rho(f) = 0$, contradicting the injectivity of $\rho$.
    \end{proof}
    Building upon this lemma, we proceed with the main argument. By Lemma \ref{lemma: lemma in the lemma of proving dif dim=1}, we may choose an integer $N \geq 1$ such that $\|\varpi^N\| < \min\{\|f\|, \|f\|_{i'}\}$ holds for all $i' \geq i$. There exists an element $g \in \varpi^N R_i^+$ such that the difference $f - g$ lies in $k[t]$. Since $\varpi^{-N}g \in R_i^+$, we have $\rho(\varpi^{-N}g) \in \widehat{\mathcal{K}}_0^+$, and therefore $\|g\| \leq \|\varpi^N\| < \|f\|$. This inequality implies that
\[
\|f\| = \|f - g\|.
\]
A similar argument shows that $\|f - g\|_{i'} = \|f\|_{i'}$ for all $i' \geq i$.

    We now analyze the structure of $f - g$. Since $f - g \in k[t]$, we may factor it as $f - g = A(t - \lambda_1)(t - \lambda_2) \cdots (t - \lambda_r)$, where $A$ and all $\lambda_u$ are elements of $k$. Choose an index $i' > i$ such that for every $u$, the point $\lambda_u$ lies outside the closed ball $B_k(x_{i'}, |z_{i'}|)$. By definition, we then have
\[
\|\lambda_u - t\| = \|\lambda_u - t\|_{i'}
\]
for every $u$. It follows that $\|f - g\|_{i'} = \|f - g\|$, which establishes the claim.

    Since $k[z] \subseteq R_\infty$ and $R_\infty$ is a field, we have $\mathcal{E}_0 = k(z) \subseteq R_\infty$. This inclusion implies that $R_\infty$ is dense in $\widehat{\mathcal{K}}_0$. Consequently, the $\varpi$-adic completion of $R_\infty^+$ is $\widehat{\mathcal{K}}_0^+$. By Theorem \ref{theo: calculate dif module by not complete}, the complete differential module $\widehat{\Omega}_{\widehat{\mathcal{E}}_0/\mathfrak{o}}$ is isomorphic to the $\varpi$-adic completion of the usual differential module $\Omega_{R_\infty^+/\mathfrak{o}}$.

    For any integer $N \geq 1$, we have the isomorphism
\[
\Omega_{R_{\infty}^{+}/\mathfrak{o}} \otimes_{\mathfrak{o}} (\mathfrak{o}/\varpi^N) = \varinjlim_i \ \Omega_{(R_i^{+}/\varpi^N)/(\mathfrak{o}/\varpi^N)},
\]
which we denote by $\mathcal{D}_N$. 
    
    We define $\mathcal{D}_N'$ as the free module of rank 1 over $R_{\infty}^{+}/\varpi^N$ with basis element $e$. Consider the $R_{\infty}^{+}/\varpi^N$-linear homomorphism 
\[
\gamma_N: \mathcal{D}_N' \to \mathcal{D}_N
\]
defined by sending $e$ to $\dif\left(\frac{t - x_i}{z_i}\right)$. 

We claim that there exists a positive integer $M$, independent of $N$, such that both $\ker(\gamma_N)$ and $\operatorname{coker}(\gamma_N)$ are annihilated by $\varpi^M$.

To verify this claim, we first observe that by definition, for each index $i$, the module $\Omega_{(R_i^{+}/\varpi^N)/(\mathfrak{o}/\varpi^N)}$ is free over $R_i^{+}/\varpi^N$ with basis $\dif\left(\frac{t - x_i}{z_i}\right)$. 

Moreover, for any indices $i \leq i'$, we have the relation
\[
\dif\left(\frac{t - x_i}{z_i}\right) = z_{i}^{-1}z_{i'} \dif\left(\frac{t - x_{i'}}{z_{i'}}\right)
\]
in $\Omega_{(R_{i'}^{+}/\varpi^N)/(\mathfrak{o}/\varpi^N)}$. This implies that in the module $\Omega_{(R_i^{+}/\varpi^N)/(\mathfrak{o}/\varpi^N)}$, we have
\[
z_{i}^{-1}z_{i'} \cdot \operatorname{Ann}_{R_i^{+}/\varpi^N}\left(\dif\left(\frac{t - x_i}{z_i}\right)\right) = 0.
\]

Therefore, if we choose $M$ such that $|\varpi^M| \leq \lim_i |z_i||z_1|^{-1}$, then $\ker(\gamma_N)$ is annihilated by $\varpi^M$. A similar argument shows that $\operatorname{coker}(\gamma_N)$ is also annihilated by the same $\varpi^M$.

This result implies that for every $N$, there exists a canonical homomorphism $\gamma_N': \mathcal{D}_N \to \mathcal{D}_N'$ satisfying
\[
\gamma_N \circ \gamma_N' = \varpi^{2M} \quad \text{and} \quad \gamma_N' \circ \gamma_N = \varpi^{2M}.
\]

Taking the projective limit over all $N$, we find that $\varprojlim_N \gamma_N$ becomes invertible after inverting $\varpi$. This establishes that
\[
\widehat{\Omega}_{\mathcal{E}_0/\mathfrak{o}}[\tfrac{1}{\varpi}]
\]
is one-dimensional, which completes the proof.
\end{proof}

\begin{proof}[Poof of Theorem \ref{main theo in intro} in characteristic $p$]
    Assume $k$ is of characteristic $p$. We need to prove the following are all equivalent:
    \begin{enumerate}[label=(\alph*)]
        \item $\calK^+$ has a flat $\delta$-lifting over $\Ainf$;
        \item $\calK^+$ is $F$-split;
        \item $\calK$ is not of type $4$.
    \end{enumerate}
    By Theorem \ref{theo: f lift to f split} and Theorem \ref{theo: type 23 has delta lifting}, it suffices to show that when $\calK$ is of type 4, $\calK^+$ is not $F$-split.

    By Theorem \ref{theo: non f split}, this is equal to show that $\calK$ is not perfect. This is a direct corollary of Lemma \ref{lemma: dim of dif type 4}.
\end{proof}

Now we turn to mixed characteristic case. 

From now on, we assume more that $k$ is of mixed characteristic $(0,p)$. Let $\frako^\flat$ be the valuation ring of $k^\flat$. Fix an element $p^\flat\in\frako^\flat$ whose untilt is $p$. Then $\xi:=p-[p^\flat]$ be a generator of $\ker\theta$.

\begin{lemma}\label{lemma: value group no extension}
    Let $E \subseteq F$ be an extension of complete non-archimedean fields. Fix a pseudo-uniformizer $\pi \in E$, and define $i_F$ as the set of invertible elements of $F/\pi$.

The following statements are equivalent:
\begin{enumerate}[label=(\roman*)]
    \item The value group of $E$ coincides with the value group of $F$.
    \item For every $\bar{f} \in F^+/\pi$, there exist $\bar{a} \in E^+/\pi$ and $\bar{u} \in i_F$ such that $\bar{f} = \bar{a}\bar{u}$.
\end{enumerate}
\end{lemma}

\begin{proof}
    Denote by $I_F$ the set of invertible elements in $F^+$. Then $i_F$ is the projection of $I_F$ in $F^+/\pi$.

    Note that the value group of $E$ is equal to the value group of $F$ if and only if each $f\in F^+$ can be written as $f=au$ such that $a\in E^+$ and $u\in I_F$. Thus, the `only if' part is easy.

    Now assume that for any $\bar f\in F^+/\pi$ can be written as $au$ for some $\bar a\in E^+/\pi$ and $\bar u\in i_F$. Let $f\in F^+$. Since $\pi$ is a pseudo-uniformizer, there exists $n\geq 1$ and $f_1\in F^+\backslash \pi F^+$ such that 
    \[f=\pi^nf_1.\]
    By our assumption, there exists $a\in E^+,\ u\in I_F$ and $g\in F^+$ such that
    \[f_1=au+\pi g.\]
    Since $f_1\not\in \pi F^+$, $a\notin \pi E^+$. Hence, as $E^+$ is a valuation ring, $a^{-1}\pi$ lies in the maximal ideal of $E^+$. Note that
    \[f_1=au+\pi g=a(u+a^{-1}\pi g)\]
    and $u+a^{-1}\pi g$ is invertible. We have $f=\pi^Nf_1$ is a product of an element of $I_F$ and an element of $E^+$. This finishes the proof.
\end{proof}

\begin{theo}\label{theo: main deform}
    Let $\mathcal{F}/k$ be an extension of complete non-Archimedean fields. The following statements hold:
\begin{enumerate}[label=(\roman*)]
    \item There exists a $(p,\xi)$-complete and $(p,\xi)$-completely flat lifting of $\mathcal{F}^+$ over $A_{\inf}$.
    \item Let $\widetilde{\mathcal{F}}^+$ be such a lifting. Then $\mathcal{F}^{+\flat} := \widetilde{\mathcal{F}}^+/p$ is a valuation ring of rank $1$ satisfying $\mathcal{F}^{+\flat}/p^{\flat} \cong \mathcal{F}^+/p$.
    \item If the value group of $\mathcal{F}$ equals the value group of $k$, then the same holds for $\mathcal{F}^{\flat}$.
\end{enumerate}
\end{theo}

\begin{proof}
    We establish each part of the theorem in sequence.

For part (i), we appeal to \cite[Corollary 3.11]{Bouis_2023}, which shows that the cotangent complex $\mathbb{L}_{(\mathcal{F}^+/p)/(\mathfrak{o}/p)}$ has projective dimension smaller than $2$. Standard deformation theory then guarantees the existence of a $(p,\xi)$-complete and $(p,\xi)$-completely flat deformation of $\mathcal{F}^+$ over $A_{\inf}$.

For part (ii), the isomorphism $\mathcal{F}^{+\flat}/p^{\flat} \cong \mathcal{F}^+/p$ is obvious. It follows directly from \cite[Lemma 3.13]{Bouis_2023}, that $\mathcal{F}^{=\flat}$ is a valuation ring. Moreover, since $\mathcal{F}^{+\flat}/p^{\flat} \cong \mathcal{F}^+/p$ is of Krull dimension $0$, $\calK^{+\flat}$ is of rank $1$.

Finally, part (iii) is an immediate consequence of Lemma \ref{lemma: value group no extension}, as $\mathcal{F}^{+\flat}/p^{\flat} \cong \mathcal{F}^+/p$.
\end{proof}

We now proceed to establish our main result in the mixed characteristic setting.

\begin{proof}[Proof of Theorem \ref{main theo in intro} in mixed characteristic]

Keep the assumptions for $k$. We need to show that the following are equivalent:
    \begin{enumerate}[label=(\alph*)]
        \item $\calK^+$ has a flat $\delta$-lifting over $\Ainf$;
        \item $\calK$ is not of type 4.
    \end{enumerate}

By Theorem \ref{theo: type 23 has delta lifting}, it suffices to demonstrate that $\mathcal{K}^+$ does not admit a flat $\delta$-lifting over $A_{\inf}$ when $\mathcal{K}$ is of type 4. Suppose, for contradiction, that such a lifting $(\widetilde{\mathcal{K}}^+, \delta)$ exists. After taking $(p,\xi)$-completion, we may assume that $\widetilde{\mathcal{K}}^+$ is $(p,\xi)$-complete.

Denote by $\calK^{+\flat}=\widetilde{\calK}^+/p$. By Theorem \ref{theo: main deform} (ii), $\widehat\Omega_{\calK^{+\flat}/\frako^\flat}\otimes_{\frako}\frako/p^{\flat}\cong\widehat{\Omega}_{\calK^+/\frako}\otimes_{\frako}\frako/p$ is almost finitely generated over $\calK^{+\flat}/p^\flat\cong \calK^+/p$ by Theorem \ref{theo: equi cond of finite dim}. Again, by Theorem \ref{theo: equi cond of finite dim}, $\dim_{\calK^\flat}\big(\widehat{\Omega}_{\calK^{\flat+}/\frako^\flat}[\frac{1}{p^\flat}]\big)$ is finite dimensional.

Applying Theorem \ref{theo: main deform}, we obtain that $\mathcal{K}^{+\flat} := \widetilde{\mathcal{K}}^+/p$ is a $p^{\flat}$-complete valuation ring with the following properties:
\begin{itemize}
    \item Denote by $\mathcal{K}^{\flat}$ the fraction field of $\mathcal{K}^{+\flat}$. Then the value group of $\mathcal{K}^{\flat}$ equals that of $k$.
    \item The residue field of $\mathcal{K}^{\flat}$ coincides with that of $k$, which is perfect.
\end{itemize}

By Theorem \ref{theo: non f split}, it follows that $\mathcal{K}^{\flat}$ must be perfect. This implies that $\mathcal{K}^{+\flat}$ is perfect as well. In particular, the Frobenius endomorphism on $\mathcal{K}^{+\flat}/p^{\flat} \cong \mathcal{K}^+/p$ is surjective, which implies that $\mathcal{K}$ is perfectoid. This contradicts Lemma \ref{lemma: dim of dif type 4}, completing the proof.
\end{proof}

\begin{rmk}
Retaining the notation from the preceding proof, we note that it remains unknown whether the fraction field of $\mathcal{K}^{+\flat}$ is topologically finitely generated over $k^{\flat}$. Consequently, we cannot directly apply the positive characteristic version of our result in this context.
\end{rmk}

\subsection{Proof of Theorem \ref{theo: main 2 f finite}}\label{pf to main 2}

In this subsection, let $k$ be an algebraically closed complete non-Archimedean field whose residue field is of characteristic $p$. Denote by $\frako$ the ring of integers of $k$. Let $\mathcal{K}/k$ be a one-dimensional analytic field.

Recall the following important consequence in commutative algebra.

\begin{theo}\label{theo: local flat finite generated free}
    Let $R$ be a local ring. Then each finitely generated flat $R$-module is free.
\end{theo}

\begin{proof}
    See \cite[Theorem 7.10]{matsumura1989commutative}.
\end{proof}

\begin{lemma}\label{lemma: f finite to dif free}
    The complete differential module $\widehat\Omega_{\calK^+/\frako}$ is free of finite rank if $\calK^+/p$ is $F$-finite.
\end{lemma}

\begin{proof}
    Assume first that the characteristic of $k$ is $p$. Let $\{e_1,e_2,\dots,e_r\}$ be a set of generators of $F_*\calK^+$ as a $\calK^+$-module. Thin means that every $f\in\calK^+$ can be written as
    \[f=\sum_{j=1}^r g_j^pe_j.\]
    Hence, in $\widehat{\Omega}_{\calK^+/\frako}$, we have
    \[\dif(f)=\sum_{j=1}^rg_j^p\dif(e_j).\]
    Let $F\subseteq \widehat{\Omega}_{\calK^+/\frako}$ be the submodule generated by $\dif(e_1),\dif(e_2),\dots,\dif(e_r)$. Then $F$ is dense in $\widehat{\Omega}_{\calK^+/\frako}$. By an approximation argument, $F=\widehat{\Omega}_{\calK^+/\frako}$. By Theorem \ref{theo: local flat finite generated free}, $\widehat{\Omega}_{\calK^+/\frako}$ is free.

    Now assume that the characteristic of $k$ is $0$. The same argument shows that $\widehat{\Omega}_{\calK^+/\frako}/p$ is finitely generated. Choose a set of elements $m_1,m_2,\dots,m_r$ whose images generate $\widehat{\Omega}_{\calK^+/\frako}/p$. By an approximation argument, we can prove that $m_1,m_2,\dots,m_r$ generate $\widehat{\Omega}_{\calK^+/\frako}$. By Theorem \ref{theo: local flat finite generated free}, $\widehat{\Omega}_{\calK^+/\frako}$ is free.
\end{proof}

\begin{lemma}\label{lemma: dif mod of type 34 not free}
    If $\calK^+$ is not of type 2, then
    \[\frakm\widehat{\Omega}_{\calK^+/\frako}=\widehat{\Omega}_{\calK^+/\frako}.\]
    In particular, $\widehat{\Omega}_{\calK^+/\frako}$ is not a free $\calK^+$-module.
\end{lemma}

\begin{proof}
    Note that in both cases, the residue field of $\calK$ is equal to the residue field of $k$.

    For any $\omega\in\widehat{\Omega}_{\calK^+/\frako}$, there exists $f_1,g_1,f_2,g_2,\dots,f_r,g_r\in \calK^+$ such that 
    \[\sum_{j=1}^r g_j\dif(f_j)-\omega\in \frakm\widehat{\Omega}_{\calK^+/\frako}.\]
    For any $j$, by the equality between the residue fields of $k$ and $\calK$, we can choose $a_j\in \frako$ such that $f_j-a_j\in \frakm\calK^+$. Hence,
    \[\sum_{j=1}^r g_j\dif(f_j)=\sum_{j=1}^r g_j\dif(f_j-a_j)\in \frakm\widehat{\Omega}_{\calK^+/\frako}.\]
    Thus, the lemma is proved.
\end{proof}

\begin{lemma}\label{lemma: faith flat + etale carry f finite}
    Let $\rho:A\to B$ be a faithfully flat, \'etale extension of algebras over $\dF_p$. Then $A$ is $F$-finite if and only if $B$ is $F$-finite.
\end{lemma}

\begin{proof}
    Since $\rho$ is \'etale. the relative Frobenius map is an isomorphism. This implies that the Frobenius map $F_B:B\to B$ can be identified as the base change 
    \[B\xrightarrow[]{\id_B\otimes F_A}B\otimes_{A,F_A}A.\]
    As $\rho$ is faithfully flat, $F_A$ is finite if and only if the base change of $F_A$ along $\rho$ is finite.
\end{proof}

\begin{proof}[Proof of Theorem \ref{theo: main 2 f finite}]
    Fix a pseudo-uniformizer $\varpi\in \frako\backslash p\frako$. By Lemma \ref{lemma: faith flat + etale carry f finite} and Theorem \ref{theo:uniformization}, we can assume that $\calK$ is a completion of $\calE_0=k(t)$ with respect to some norm on $\calE_0$.

    First, we prove that if $\calK$ is of type $2$, then $\calK^+/p$ is $F$-finite. This time, $\calE_0^+/p$ is isomorphic a localization of a polynomial algebra of $\frako/p$, and $\calK^+/p$ is the $\varpi$-adic completion of $\calE_0^+$. Hence, $\calE_0^+/p$ is $F$-finite, so is its completion $\calK^+/p$.

    Next, we prove that when $\calK$ is of type 3 or type 4, $\calK^+/p$ is not $F$-finite. Indeed, by Lemma \ref{lemma: dif mod of type 34 not free}, $\widehat{\Omega}_{\calK^+/\frako}$ is not free. Then, by Lemma \ref{lemma: f finite to dif free}, $\calK^+/p$ is not $F$-finite.
\end{proof}

\newpage
\appendix

\section{Appendix: Description of the ring of integers of Gauss norms}\label{appd:A}

The following lemma may be well known for experts but we could not find a proof. So we write it down here.

\begin{lemma}\label{lemma: Gabber's calculation of val rings}
    Assume more that $k$ is algebraic colsed. Let $r\notin \Gamma$ and $\calO_r$ be the valuation ring of the Gauss norm on $k(t)$ with center $0$ and radius $r$. Suppose we are given two sequences $\{z_i:i=1,2,\dots\}$ and $\{w_i:i=1,2,\dots\}$ in $k$ satisfying the following:
    \begin{enumerate}
        \item For any $i$, $|w_i|<|w_{i+1}|<r<|z_{i+1}|<|z_i|$;
        \item $\lim_{i\to \infty}|z_i|=r=\lim_{j\to\infty}|w_j|$.
    \end{enumerate}
    For each $i$, define $R_i$ as the localization of the subring $\mathfrak{o}[\frac{w_i}{t},\frac{t}{z_i}]\subseteq k[t^{\pm1}]$ at the prime generated by $\frac{w_i}{t}$, $\frac{t}{z_i}$ and $\mathfrak{m}$. Then, $\calO_r$ is the union of all $R_i$.
\end{lemma}

\begin{proof}
    Since $k$ is algebraically closed, $\Gamma$ is a $\dQ$-vector space. Hence, $r^n\notin \Gamma$ for any $0\neq n\in \dZ$. Denote by $\|\bigcdot\|_r$ the Gauss norm with center $0$ and radius $1$.

    Note that $R_i\subseteq R_{i+1}\subseteq\calO_r$ for all $i$. We find $\bigcup_iR_i$ is a subring of $\calO_r$.

    Let $P_\infty$ be the union of all $\mathfrak{o}[\frac{w_i}{t},\frac{t}{z_i}]$. Then, $P_\infty=k[t^{\pm1}]\cap \calO_r$. Indeed, for each $f=\sum_{j=-d}^da_jt^j\in\calO_r$, we have
    \[|a_j|r^j\le 1\]
    for all $j$. Since for any $j\neq 0$, $|a_j|r^j\notin \Gamma$, $|a_j|r^j<1$. Hence, as there are only finitely many non-zero $a_j$, there exists $i$ such that $f\in R_i$.

    Let $M_i$ be the ideal of $R_i$ generated by $\frac{w_i}{t}$, $\frac{t}{z_i}$ and $\mathfrak{m}$. A similar argument in the previous paragraph shows that the union of all $R\backslash M_i$ is equal to
    \[\{f\in k[t^{\pm1}]:\|f\|_r=1\}\]

    Since taking filtered colimit commutes with localization, we only need to check that the localization of $P_\infty$ at the prime\[\{f\in k[t^{\pm1}]:\|f\|_r<1\}\]
    is equal to $\calO_r$. Denote by $P'_\infty$ this localization.

    Obviously, $k(t)$ is equal to the fractional field of $P_\infty$. Hence, any element in $\calO_r$ can be written as $\frac{f}{g}$. Let $g=\sum_{j=-d}^da_jt^j$. By the definition of $\|\bigcdot\|_r$, there exists $j$ such that $|a_j|r^j=\|g\|_r$.

    First assume that $j=0$, then by definition, $a_0^{-1}g\in k[t^{\pm1}]$ has Gauss norm $1$. Since $\frac{f}{g}\in \calO_r$, we have $\|a_0^{-1}f\|_r=|a_0|^{-1}\|f\|_r\leq 1$. Hence,
    \[\frac{f}{g}=\frac{a_0^{-1}f}{a_0^{-1}g}\]
    lies in $P_\infty'$

    It remains to prove for $j\neq 0$. This time, $\|g\|_r<1$ (or $|a_0|=1$ and this causes $j=0$). Hence, we can choose $x\in k^\times$ such that $\|g\|_r<\|(xt)^j\|_r<1$ by the density of $\Gamma$ in $\dR_{>0}$. Then we have
    \[\frac{f}{g}=\frac{(xt)^{-j}f}{(xt)^{-j}g}\]
    and this reduces to the $j=0$ case.
\end{proof}

\bibliographystyle{alpha}
\bibliography{ref}

\end{document}